\documentclass[11pt, a4paper]{amsart}
\usepackage{amsmath, amssymb, amsthm}
\usepackage[english]{babel}
\usepackage[latin1]{inputenc}
\usepackage{xy}
\usepackage{graphicx}
\usepackage{verbatim}
\usepackage{enumerate}
\usepackage{url}
\usepackage{subfig}
\usepackage{float}

\newcommand{\maxdim}{n}

\newcommand{\ud}{\mathrm{d}}

\newcommand{\Oml}{\underline{\Omega}}
\newcommand{\Omu}{\overline{\Omega}}
\newcommand{\Om}{\Omega}

\newcommand{\OmlD}[1]{\Oml_{\text{df}(#1)}}
\newcommand{\OmuD}[1]{\Omu_{\text{df}(#1)}}
\newcommand{\OmD}[1]{\Om_{\text{df}(#1)}}

\newcommand{\Th}{\Theta}

\newcommand{\Omk}{\Om_k}
\newcommand{\Omj}{\Om_j}

\newcommand{\sm}{\setminus}

\newcommand{\e}{\varepsilon}
\newcommand{\vphi}{\varphi}

\newcommand{\R}{\mathbb{R}}
\newcommand{\Rn}{\R^\maxdim}
\newcommand{\Rno}{\R^{\maxdim_1}}
\newcommand{\Rnt}{\R^{\maxdim_2}}
\newcommand{\Z}{\mathbb{Z}}
\newcommand{\N}{\mathbb{N}}
\newcommand{\C}{\mathcal{C}}

\newcommand{\lmax}{\lambda_{\text{max}}}

\newcommand{\Bm}{\beta}

\newcommand{\gh}{\hat{g}}

\newcommand{\Tor}{\mathbf{T}}
\newcommand{\Torn}{\Tor^\maxdim}
\newcommand{\Tornn}{\Tor^{2\maxdim}}

\newcommand{\Dn}{\Delta_{\maxdim}}

\newcommand{\Sg}{\mathbf{S}}
\newcommand{\Sk}[1]{\Sg_{#1}}
\newcommand{\Sn}{\Sg_\maxdim}

\newcommand{\Ai}{A_{\textrm{inc}}}
\newcommand{\Ac}{A_{\textrm{coh}}}

\newcommand{\abs}[1]{\left|#1\right|}
\newcommand{\tabs}[1]{\big|#1\big|}
\newcommand{\norm}[1]{\left\|#1\right\|}

\newcommand{\rset}[2]{\left\lbrace\, #1\,\left|\;#2\right.\right\rbrace}

\newcommand{\set}[2]{\rset{#1}{#2}}
\newcommand{\tset}[2]{\big\lbrace #1\,\big|\;#2\big\rbrace}
\newcommand{\sset}[1]{\left\lbrace #1\right\rbrace}
\newcommand{\tsset}[1]{\big\lbrace #1\big\rbrace}

\newcommand{\Cinf}{C^\infty}

\newcommand{\fleft}{\!\left}

\author{Christian Bick and Peter Ashwin}

\address{Centre for Systems, Dynamics and Control and Department of Mathematics, University of Exeter, Exeter EX4 4QF, UK}

\title[Persisting Chaotic Weak Chimeras]{Chaotic Weak Chimeras and their Persistence in Coupled Populations of Phase Oscillators}

\newtheorem{prop}{Proposition}
\newtheorem{lem}{Lemma}
\newtheorem{thm}{Theorem}
\newtheorem{cor}{Corollary}
\theoremstyle{definition}
\newtheorem{defn}{Definition}
\theoremstyle{remark}
\newtheorem{rem}{Remark}

\newcommand{\imagescaling}{0.9}

\begin{document}


\begin{abstract}
Nontrivial collective behavior may emerge from the interactive dynamics of many oscillatory units. Chimera states are chaotic patterns of spatially localized coherent and incoherent oscillations. The recently-introduced notion of a weak chimera gives a rigorously testable characterization of chimera states for finite-dimensional phase oscillator networks. In this paper we give some persistence results for dynamically invariant sets under perturbations and apply them to coupled populations of phase oscillators with generalized coupling. In contrast to the weak chimeras with nonpositive maximal Lyapunov exponents constructed so far, we show that weak chimeras that are chaotic can exist in the limit of vanishing coupling between coupled populations of phase oscillators. We present numerical evidence that positive Lyapunov exponents can persist for a positive measure set of this inter-population coupling strength.
\end{abstract}

\maketitle


\section{Introduction}
The emergence of collective behavior is a fascinating effect of the interaction of oscillatory units~\cite{Strogatz2000, Pikovsky2003, Strogatz2004} and coupled phase oscillators~\cite{Ashwin1992} serve as paradigmatic mathematical models to study these dynamical states in many system of interest ranging from technology to neuroscience~\cite{Acebron2005, Tchistiakov1996, Ashwin2015}.

In addition to global synchronization, the emergence of locally synchronized coherence-incoherence patterns---commonly known as chimera states~\cite{Kuramoto2002, Abrams2004}---has received a lot of attention in recent years~\cite{Panaggio2015a}. These are particularly of interest where the patterns have broken symmetry with respect to the networks. Similar states have been shown to exist in real-world experiments~\cite{Tinsley2012, Hagerstrom2012, Martens2013} and may be exploited for applications~\cite{Bick2014a}. While chimera states are typically studied at or near the limit of infinitely many oscillators as stationary patterns of the phase density distribution~\cite{Omel'chenko2013}, they have generally only been described phenomenologically.

Ashwin and Burylko recently introduced a testable definition of a chimera state in the context of finite networks of indistinguishable phase oscillators---a weak chimera~\cite{Ashwin2014a}. Weak chimeras are defined for oscillators where the phases $\vphi_k\in \Tor=\R/2\pi \Z$ evolve according to
\begin{equation}
\frac{\ud\vphi_k}{\ud t} = \omega - \frac{1}{n} \sum_{j=1}^n H_{kj} g(\vphi_k-\vphi_j)
\label{eq:COsc}
\end{equation}
in terms of partial frequency synchronization on trajectories; here~$H_{kj}$ gives
the network topology (respecting a permutation symmetry that acts transitively on the indices of the oscillators) and~$g(\phi)$ is the generalized coupling (phase interaction) function. Such weak chimeras cannot appear in fully symmetric globally coupled phase oscillator networks or in any system with three or fewer oscillators. For coupling functions with more than one Fourier mode, 
there are examples of weak chimeras in systems of four oscillators that are relative periodic orbits, relative quasiperiodic orbits for six oscillators, and weak chimeras of heteroclinic type in a system of ten oscillators~\cite{Ashwin2014a}. Note that coupling functions with multiple Fourier modes are not necessary for the occurrence of attracting weak chimeras: they can be found even for Kuramoto--Sakaguchi coupling $g(\phi)=\sin(\phi+\alpha)$ in such a system with only four oscillators~\cite{Panaggio2015b}.

However, the weak chimeras in \cite{Ashwin2014a} fail to capture one important dynamical feature expected of chimera states in higher-dimensional phase oscillator network: namely that they are chaotic. Moreover, the definition assumes existence of limiting frequency differences, which may not be the case for general trajectories even for a chaotic invariant set possessing a natural measure. Numerical investigations indicate that the chimeras in certain nonlocally coupled rings of oscillators may exhibit positive maximal Lyapunov exponents~\cite{Wolfrum2011a}---while for attractive coupling, they may appear only as transients for typical initial conditions~\cite{Wolfrum2011b}. By contrast, weak chimeras constructed in~\cite{Ashwin2014a, Panaggio2015b} have Lyapunov exponents that are presumably nonpositive, and can be attractors or repellers. A natural question is whether it is possible find ``chaotic'' weak chimeras in relatively small finite networks of indistinguishable oscillators---are there weak chimeras whose dynamics have positive maximal Lyapunov exponents for typical orbits?

In this paper we construct systems that exhibit such chaotic weak chimera states. It is already known that positive maximal Lyapunov exponents may arise in a fully symmetric phase oscillator system if the coupling function is chosen appropriately~\cite{Bick2011}. We use this to show the existence of weak chimeras for weakly coupled clusters of oscillators with positive Lyapunov exponents in the limit of vanishing coupling. Then we show numerically that the positive Lyapunov exponents persist for nonvanishing coupling. 
Consequently, the study of chaotic weak chimeras sheds some light on the dynamics of regular chimera states; indeed one might conjecture that ``typical'' weak chimeras will be chaotic in all but the smallest and simplest systems.

The paper is organized as follows. In Section~\ref{sec:Persistence} we discuss some results on persistence of invariant sets for general nonautonomously perturbed dynamical systems and their consequences for weakly coupled product systems. In Section~\ref{sec:Prelims} we review and 
make a simple generalization of the definition of a weak chimera state to cases where the frequency difference may not exist on certain trajectories, but may only have upper and lower bounds. In Section~\ref{sec:CWC} we apply the results of Section~\ref{sec:Persistence} to prove the existence of weak chimeras for systems that are weakly coupled populations of phase oscillators. We numerically investigate the maximal Lyapunov exponents of these weak chimeras in Section~\ref{sec:Numerics} and observe that the maximal Lyapunov exponent stays positive for a large (presumably positive measure) set of coupling strengths between clusters. Finally, we give some concluding remarks.

\section{Invariant sets and their persistence under perturbations}
\label{sec:Persistence}

In this section we give a general result for the persistence of absorbing sets under general nonautonomous bounded perturbations. Applied to weakly coupled systems, this yields a persistence result for dynamically invariant sets. We will use this result in the following section to show the existence of weak chimeras. The assumptions we make are sufficient to prove the results but not necessary as far as we can determine.

Some notation we will use throughout this manuscript. Let~$\R$ denote the field of real numbers and~$\Rn$ then $\maxdim$-dimensional
vector space over~$\R$. If $x,y\in\Rn$ then we denote their scalar product by $x\cdot y$. If $F:\Rn\to\R^m$ then let~$F'$ denote the (total) derivative and for  $x:\R\to\Rn, t\mapsto x(t)$ write $\dot x=\frac{\ud x}{\ud t}$. A (time-dependent) smooth vector field 
$f:\Rn\times\R\to\Rn$ defines an ordinary differential equation 
\begin{equation}\label{eq:ODE}
\dot x = f(x, t).
\end{equation}
If the system is autonomous, that is, $f(x, t)=f(x)$, then for $x\in\Rn$ the associated~$\alpha$ and $\omega$-limit sets are closed and invariant with respect to the dynamics.

\subsection{Persistence of absorbing regions}

\begin{defn}
A set~$R\subset\Rn$ with compact smooth boundary is \emph{forward absorbing} for~\eqref{eq:ODE} if there is a differentiable function $W:\Rn\to\R$ such that $\partial R = \set{x\in\Rn}{W(x)=0}$ and $W'(x)\cdot f(x, t)<0$ for all $x\in\partial R$, $t\in\R$. A set is \emph{backward absorbing} if it is forward absorbing for~$-f$.
\end{defn}

In other words, an absorbing set has a boundary given by the zeros of some differentiable function and the vector field~$f(x, t)$ points either inward or outward everywhere on this boundary. Clearly, if~\eqref{eq:ODE} is autonomous and~$R$ is a forward or backward absorbing set then~$R$ contains an invariant set of the dynamics.

We now show persistence of absorbing regions for an autonomous system
\begin{align}
\label{eq:aut}
\dot{x}&=f(x)
\end{align}
subject to bounded nonautonomous perturbations. More precisely, consider
\begin{align}
\label{eq:nonaut}
\dot{x}&=f(x)+\e g(x,t)
\end{align}
with~$\e\geq 0$ and $g:\Rn\times\R\to\Rn$ smooth with $\norm{g(x,t)}<M<\infty$ for 
all~$t>0$ and $x\in\Rn$. 

\begin{lem}\label{lem:AbsorbingRegion}
Suppose that~$R$ is forward absorbing for $\e=0$. Then there is an $\e_0>0$ such that whenever $0\leq\e<\e_0$ the set $R$ is forward absorbing for the dynamics of \eqref{eq:nonaut}.
\end{lem}

\begin{proof}
We just have to check that $W'(x)\cdot (f(x)+\e g(x,t))<0$ for all $x\in \partial R$ and $t>0$. Since $W'(x)\cdot f(x)$ is negative on the compact surface~$\partial R$ it has a lower bound, i.e., there is a $\xi>0$ such that $W'(x)\cdot f(x)<-\xi$ for all $x\in \partial R$. Thus we have for $x\in\partial R$
\[W'(x)\cdot(f(x)+\e g(x,t)) \leq -\xi + \e \max\set{W'(x)\cdot g(x,t)}{x\in \partial R, t\in\R}\]
so if we pick
\[
\e_0= \frac{\xi}{\max \tset{M\norm{W'(x)}}{x\in \partial R}}
\]
then for each $\e<\e_0$ and $t>0$ we have $W'(x)\cdot(f(x)+\e g(x,t))<0$ as required.
\end{proof}

\begin{defn}
Let~$A$ be a compact invariant set for the dynamics of~\eqref{eq:aut}.
\begin{enumerate}[(a)]
\item
A nonnegative differentiable function $V:\Rn\to\R$ is a \emph{Lyapunov function for~$A$} if $V(x) = 0$ for all $x\in A$ and $\dot{V}(x) = V'(x)\cdot f(x)<0$ for $x\not\in A$.
\item If~$V$ is a Lyapunov function for~$A$ then we call~$A$ \emph{sufficiently stable}.
\item If $V$ is a Lyapunov function for~$A$ for $-f$ then we call~$A$ \emph{sufficiently unstable}.
\end{enumerate}
\end{defn}

One can check that sufficient stability of a set implies asymptotic stability, though the converse only holds under additional assumptions; see for example \cite{Michel2015}. 

We now show that Lyapunov functions give absorbing regions arbitrarily close to a compact invariant set~$A$ of the unperturbed dynamics as long as the perturbation is sufficiently small. We write $B_{\delta}(A)$ to denote a $\delta$-neighborhood of~$A$.

\begin{prop}
\label{thm:NonautPerturb}
Suppose that~$A$ is a compact invariant sufficiently stable (or unstable) set for the unperturbed dynamics~\eqref{eq:aut} and $\norm{g(x,t)}<M$ for all~$x\in\Rn$ and $t>0$. For any $\delta>0$ there is a 
compact set~$R$ with 
\[
A\subset R \subset B_{\delta}(A)
\] 
and an~$\e_0$ such that whenever $0\leq\e<\e_0$ the set~$R$ is an absorbing region for the dynamics of the perturbed system~\eqref{eq:nonaut}.
\end{prop}

\begin{proof}
Sufficient persistence implies that there is a Lyapunov function~$V$.
For $\eta>0$ define
\[
R_\eta:= \set{x\in \Rn}{V(x)\leq\eta}.
\]
For given $\delta>0$ choose~$\eta>0$ such that 
$A\subset R_{\eta} \subset B_{\delta}(A)$. 
Since $V$ is a Lyapunov function the set~$R=R_{\eta}$ is an absorbing region for the dynamics of~\eqref{eq:aut}. Lemma~\ref{lem:AbsorbingRegion} now implies that there is 
an~$\e_0$ such that~$R$ is also absorbing for the dynamics 
of~\eqref{eq:nonaut}.
\end{proof}

\subsection{Persistence of invariant sets in product systems}

Nonautonomous perturbations may arise from weak coupling of two dynamical systems.

\subsubsection{Weak forcing respecting an invariant set}

Let $x\in\Rno$, $y\in\Rnt$ and consider the product system
\begin{equation}
\label{eq:coupledforced}
\begin{aligned}
\dot{x}&=f(x)+ \e g(x,y)\\
\dot{y}&=h(x,y).
\end{aligned}
\end{equation}
Suppose that $V\subset\Rnt$ is a compact set such that $\Rno\times V$ is dynamically invariant for all $\e\geq 0$.

\begin{thm}
\label{thm:PersistenceFactor}
Let $A\subset\Rno$ be a compact and sufficiently stable (or sufficiently unstable) set for $\dot{x}=f(x)$ and suppose $\norm{g(x,y)}\leq M<\infty$ for all $(x,y)$. For any $\delta>0$ there exists an $\e_0>0$ such that whenever $0\leq\e<\e_0$ there is a dynamically invariant set $A_{\e} \subset B_\delta(A)\times V$ of the perturbed system~\eqref{eq:coupledforced}.
\end{thm}

\begin{proof}
For given $\delta >0$, Proposition~\ref{thm:NonautPerturb} yields an~$\e_0$ and $R\subset B_\delta(A)$ such that~$R$ is absorbing for all $0\leq\e<\e_0$.

Suppose that~$A$ is sufficiently stable. The $\omega$-limit set $A_{\e}=\omega(x, y)$ of $(x,y)\in R\times V$ is dynamically invariant and we have $D\subset B_\delta(A)\times V$ since~$R$ is absorbing. If~$A$ is sufficiently unstable take the $\alpha$-limit set instead.
\end{proof}

\subsubsection{Weakly coupled product systems}

We now present a similar result for weakly coupled systems, 
\begin{equation}
\label{eq:coupled}
\begin{aligned}
\dot{x}&=f_1(x)+\e g_1(x,y)\\
\dot{y}&=f_2(y)+\e g_2(x,y)
\end{aligned}
\end{equation}
with $x\in\Rno$, $y\in\Rnt$ and $f_\ell,g_\ell$ smooth such that $\norm{g_\ell}<M<\infty$, $\ell=1, 2$. We refer to $\dot{x}=f_1(x)$ and $\dot{y}=f_2(y)$ as the \emph{uncoupled factors} of~\eqref{eq:coupled}.

\begin{thm}\label{thm:Products}
Suppose that the uncoupled factors of \eqref{eq:coupled} have sufficiently stable attractors $A_1$ and $A_2$. Let $A=A_1\times A_2$. Then for any $\delta>0$ there exists an~$\e_0>0$ such that for all $0\leq \e<\e_0$ there is an invariant set $A_\e\subset B_\delta(A)$ for the dynamics of~\eqref{eq:coupled}. The same holds if~$A_1$ and~$A_2$ are sufficiently unstable. 
\end{thm}

\begin{proof}
We repeat the same argument as for Theorem~\ref{thm:NonautPerturb} 
for the Lyapunov function $V(x,y):=V_1(x)+V_2(y)$ to show that 
for any $\delta>0$ there is an~$\e_0$ and an absorbing region~$R$ 
with
\[
A_1\times A_2 \subset R \subset B_{\delta}(A_1\times A_2).
\]
for the dynamics of \eqref{eq:coupled} for all $0\leq\e<\e_0$.
\end{proof}

\begin{rem}
Since Theorems~\ref{thm:PersistenceFactor} and~\ref{thm:Products} are local results, they clearly generalize the case where the assumptions are satisfied on a sufficiently large neighborhood of the set of interest.

In fact, the existence of absorbing regions $R_1\supset A_1$, $R_2\supset A_2$ 
suffices to show the existence of an invariant set $A_{\e}\subset R_1\times R_2$
for sufficiently small~$\e$.
By contrast, a Lyapunov function allows one to construct invariant sets arbitrarily close to the product $A_1\times A_2$.
\end{rem}

\subsubsection{Dynamics on the invariant sets}
Note that the dynamics on $A_\e$ described in Theorems~\ref{thm:PersistenceFactor} and~\ref{thm:Products} may be qualitatively very different than the dynamics on~$A$. In fact, even if one assumes appropriate contraction and expansion properties such that~$A$ is uniformly normally hyperbolic and~$A_\e$ is the invariant set that arises when the vector field is perturbed, one cannot expect that the dynamics on~$A$ and $A_\e$ are conjugate~\cite{Fenichel1972}.

Persistence relates to the upper semicontinuity of attractors. In fact, many attractors of dynamical systems will be persistent in the sense of upper semicontinuity of the attracting set to a wide class of perturbations that may include nonautonomous perturbations. A more precise statement of this requires framing assumptions on the properties of the attractor and of the class of perturbations. For example, \cite{Afraimovich1998} give some conditions that ensure persistence of attractors to nonautonomous perturbations, with applications to synchronization problems, while semicontinuity of attractors to discretization is clearly an important topic in the study of numerical approximations of dynamical systems \cite{Stuart1994}.

\section{Weak chimeras for coupled phase oscillators}
\label{sec:Prelims}

\subsection{Frequency synchronization and weak chimeras}

We now consider the dynamics of~$\maxdim\in\N$ phase oscillators 
where oscillator~$k$ is characterized by its state~$\vphi_k\in\Tor$. 
Rather than restricting to systems of the form \eqref{eq:COsc}, we 
first consider a more general setting.
Let~$F=(F_1, \dots, F_\maxdim)$ be a smooth vector field on the
torus~$\Torn$. The evolution of the $k$th oscillator is given by
\begin{equation}\label{eq:Dyn}
\dot\vphi_k = F_k(\vphi_1,\ldots,\vphi_n).
\end{equation}

\subsubsection{Average angular frequency intervals and weak chimeras}

For any $\vphi^0\in\Torn$ and $T>0$ let us define the {\em average angular frequency of oscillator~$k$ on the time interval} $[0, T]$ by
\[
\langle F_k \rangle_{T, \vphi^0} := \frac{1}{T}\int_{0}^{T} F_k(\vphi(t))\ud t,
\] 
where $\vphi(t)$ is the solution of~\eqref{eq:Dyn} with initial condition $\vphi(0)=\vphi^0$. More generally, for $A\subset \Torn$ a compact and forward invariant set under the flow determined by~$F$ we define
\begin{align}
\Oml_k(F, A) &:= \inf_{\vphi\in A}\liminf_{T\to\infty}\langle F_k \rangle_{T, \vphi},\\
\Omu_k(F, A) &:= \sup_{\vphi\in A}\limsup_{T\to\infty}\langle F_k \rangle_{T, \vphi}
\end{align}
to describe the minimal and maximal average frequency on~$A$. Similarly, define the angular frequency differences
\begin{align}
\OmlD{k;j}(F, A) & := \inf_{\vphi\in A}\liminf_{T\to\infty} \big(\langle F_k \rangle_{T, \vphi}-\langle F_j \rangle_{T, \vphi}\big),\\
\OmuD{k;j}(F, A) & := \sup_{\vphi\in A}\limsup_{T\to\infty}\big(\langle F_k \rangle_{T, \vphi}-\langle F_j \rangle_{T, \vphi}\big)
\end{align}
where we use the subscript $\text{df}(k;j)$ to indicate the frequency difference between oscillators~$k$ and~$j$.

\begin{rem}
If~$F$ describes the dynamics of weakly coupled phase oscillators then $\Oml_k(F, A), \Omu_k(F, A)$ converge under fairly weak assumptions on smoothness of the dynamics~\cite{Karabacak2009}. There are various alternative ways to express the interval of angular frequency differences, for
example one can use a continuous version of~\cite[Prop.~2.1]{Jenkinson2006} to write
\begin{align}
\OmlD{k;j}(F, A) & = \inf_{\mu\in M_A}\int \langle F_k \rangle_{T, \vphi}-\langle F_j \rangle_{T, \vphi}\, d\mu,\\
\OmuD{k;j}(F, A) & = \sup_{\mu\in M_A}\int \langle F_k \rangle_{T, \vphi}-\langle F_j \rangle_{T, \vphi}\, d\mu.
\end{align}
where~$M_A$ is the set of ergodic probability measures invariant under the flow generated by $F$ that are supported on~$A$.
\end{rem}

\begin{defn}
The \emph{average angular frequency interval} of the $k$th oscillator on~$A$ is given by
\begin{equation}\label{eq:Freqs}
\Omk(F, A) := \left[\Oml_k(F, A), \Omu_k(F, A)\right]\subset\R
\end{equation}
and the \emph{average angular frequency difference interval} between
oscillators $k, j$ is given by
\begin{equation}\label{eq:FreqDiffs}
\OmD{k;j}(F, A) := \left[\OmlD{k;j}(F, A), \OmuD{k;j}(F, A)\right]\subset\R.
\end{equation}
We say the oscillators $k, j$ are \emph{frequency synchronized} on~$A$ if 
\[\OmD{k;j}(F, A) = \sset{0}.\]
\end{defn}

Note that even if $\Omk(F, A)$ and $\Omj(F, A)$ are intervals with interior, it is possible that oscillators $k, j$ are frequency synchronized. We now define weak chimeras by frequency synchronization---this is a generalization of the definition in \cite{Ashwin2014a}.

\begin{defn}
A compact, connected, chain-recurrent forward-invariant set~$A$ is 
a \emph{weak chimera} if for there are distinct oscillators 
$k, j, l$ such that
\begin{align} 
\OmD{k;j}(F, A) &= \sset{0},\\
\Om_l(F, A)\cap\Omk(F, A)&=\emptyset.
\end{align}
\end{defn}

As in~\cite{Ashwin2014a} we assume minimal conditions on~$A$ that are satisfied if it is an $\omega$-limit set for the dynamics. The definition in \cite{Ashwin2014a} can be seen a special case of this definition in the case that there is convergence of the average frequency differences: $\OmD{k;j}(F, A)=\OmD{k;j}(F, A)$.

\subsubsection{Average angular frequencies under perturbations}
What is the effect of a bounded perturbation on the average frequencies of oscillators with dynamics given by~\eqref{eq:Dyn}? More precisely, if $\e>0$, $M \geq 0$, $Y_k:\R\to\R$ is smooth with $\abs{Y_k(t)}\leq M<\infty$ for $k\in\sset{1, \dotsc, \maxdim}$ and all $t>0$, we are interested in the average angular frequencies of the dynamics determined by the vector field 
\begin{align}
\label{eq:Feps}
\dot{\vphi}_k = F^{(\e)}_k(\vphi, t) =  F_k(\vphi) + \e Y_k(t).
\end{align}

We now give a (very) approximate bound for the 
average angular frequencies for dynamics with respect to the 
perturbed vector field~$F^{(\e)}$. For $A\subset\Torn$ 
define the interval
\begin{equation}
\label{eq:Bound}
N_k(F, A) := \Big[\inf_{x\in A}F_k(x), \sup_{x\in A}F_k(x)\Big].
\end{equation} 
Note that if~$A$ is compact, so is~$N_k(F, A)$.

\begin{lem}\label{lem:CloseCoarseBound}
Let~$A$ be a compact dynamically invariant set for~\eqref{eq:Feps} 
with $\e=0$ and let $\delta\geq 0$, $\e_0>0$ be given. Then there 
is an $\eta\geq 0$ such that for any $0\leq\e<\e_0$, if there is a 
compact dynamically invariant set $A_\e\subset B_\delta(A)$ 
for~\eqref{eq:Feps} then 
\begin{align*}
\Omk(F^{(\e)}, A_\e)&\subset B_{\eta}(N_k(F, B_\delta(A))).
\end{align*}
\end{lem}

\begin{proof}
The statement follows from explicit integral estimates; we give here the 
upper bound---the lower bound is obtained analogously.

If~$\vphi(t)$ is a solution of~\eqref{eq:Feps} with 
$\vphi(0)=\vphi^0\in A_\e$ set $D=\set{\vphi(t)}{t\geq 0}$.
For any $T>0$ we have
\begin{align*}
\langle F^{(\e)}\rangle_{T, \vphi^0} &= \frac{1}{T}\int_{0}^{T}F^{(\e)}_k(\vphi(t))\ud t
= \frac{1}{T}\int_{0}^{T}F_k(\vphi(t)) + \e Y_k(t)\ud t\\
& \leq \sup_{\psi\in D} {F_k(\psi)} + \e M \\
& \leq \sup_{\psi\in {B_\delta(A)}} {F_k(\psi)} + \e_0 M.
\end{align*}
Thus, for $\eta=\e_0 M$ we have 
\[\Omu_k(F^{(\e)}, A_\e)=\sup_{\vphi^0\in A_\e}\limsup_{T\to\infty}\langle F^{(\e)}\rangle_{T, \vphi^0}\leq \sup_{\psi\in {B_\delta(A)}} {F_k(\psi)} + \eta\]
which proves the assertion.
\end{proof}

\begin{rem}
In particular, Lemma~\ref{lem:CloseCoarseBound} implies that 
$\Omega_k(F,A)\subset N_k(F,A)$. While Lemma~\ref{lem:CloseCoarseBound} 
suffices for our purposes, using continuity of~$F$ one can prove a 
stronger statement:
for any $\eta>0$ there exists a $\delta>0$ and $\e_0>0$ such that 
for any compact and invariant~$A_\e\subset B_\delta(A)$ 
for~\eqref{eq:Feps} with $0\leq\e<\e_0$ 
we have
$\Omk(F^{(\e)}, A_\e)\subset B_{\eta}(N_k(F, A))$.

Moreover, if one assumes ergodicity and the existence of suitable invariant measures then the time averages may be replaced by spatial averages.
If in addition the measure deforms nicely under the perturbation of the vector field then we can obtain better approximation of the frequencies of the perturbed system than the coarse bound given in Lemma~\ref{lem:CloseCoarseBound}.
\end{rem}

\subsection{Networks of globally coupled phase oscillators}

The notion of frequency synchronization and weak chimeras applies to systems of identical and diffusively coupled phase oscillators. We define
$\Th:\Torn\times\Torn\to\Rn$ by
\begin{equation}
\label{eq:Theta}
\Th_k(\vphi, \psi) := \frac{1}{\maxdim}\sum_{j=1}^{\maxdim}g(\vphi_k-\psi_j).
\end{equation}
where $g:\Tor\to\R$ is the $2\pi$-periodic coupling (phase interaction) function. Now consider the system of~$n$ globally coupled identical phase oscillators 
\begin{equation}\label{eq:OscVF}
\dot\vphi_k = F_k(\vphi) = \omega+\Th_k(\vphi, \vphi)
\end{equation}
with $\Th$ defined as in \eqref{eq:Theta}. We may assume $\omega=0$ without loss of generality.

The system~\eqref{eq:OscVF} is $\Sn\times\Tor$-equivariant~\cite{Ashwin1992}; 
the group~$\Sn$ of permutations acts by permuting the indices of 
the oscillators and the continuous symmetry~$\Tor$ acts by shifting 
the phase of all oscillators. As a consequence, both the 
diagonal
\begin{align}
\Dn &= \set{(\vphi_1, \dotsc, \vphi_\maxdim)\in\Torn}{\vphi_1 = \dotsb = \vphi_\maxdim}\\
\intertext{and the open~\emph{canonical invariant region}}
\label{eq:CIR}
\C &:= \set{(\vphi_1, \dotsc, \vphi_\maxdim)}{\vphi_1 < \dotsb < \vphi_n < \vphi_1+2\pi},
\end{align}
are dynamically invariant with respect to the dynamics 
of~\eqref{eq:OscVF}. The latter is bounded by codimension one invariant 
subspaces corresponding to $(n-1)$-cluster states that all intersect
at~$\Dn$.

\begin{rem}
The preservation of phase ordering implies that weak chimeras cannot exist for~\eqref{eq:OscVF} for any~$\maxdim$ even for our more general definition of a weak chimera~\cite{Ashwin2014a}.
\end{rem}

For some specific solutions, the average angular frequencies are easy 
to calculate:

\begin{enumerate}[(a)]
\item
If~$A=\sset{\vphi^*}$ is a fixed point of~\eqref{eq:OscVF} (relative to the continuous group action) then we have $\Omk(F, A) = \tsset{\frac{1}{\maxdim}\sum_jg(\vphi^*_n-\vphi^*_j)}$. 
\item
If~$A$ is an arbitrary (relative) periodic orbit $\vphi(t)$ with period~$P>0$ then $\Omk(A) = \tsset{\frac{1}{P}\int_0^P F_k(\vphi(t))\ud t}$.
\item
We have $\Omk(F, \Dn) = \sset{g(0)}$ where~$\Dn$ is either a 
continuum of fixed points (if $g(0)=0$) or the periodic orbit 
$\vphi(t)=(g(0)t, \dotsc, g(0)t)$. In particular, all oscillators are 
frequency synchronized.
\end{enumerate}

The following two observations become important in the next section; they
assert that changing the coupling function~$g$ 
locally at zero affects the angular frequencies on~$\Dn$ but not the 
dynamics of~\eqref{eq:OscVF} on compact~$A\subset\C$.

First, Lemma~\ref{lem:CloseCoarseBound} gives information about
the average angular frequencies of almost fully (phase) synchronized 
oscillators. Write $Z^{(\e)}=F+\e Y$ for a small perturbation of~$F$ 
as in~\eqref{eq:Feps} and set $D=\set{\vphi(t)}{t\geq 0}$ for a 
solution~$\vphi$ for the flow of~$Z^{(\e)}$ that stays close to~$\Dn$.
We have $\Omk(F, \Dn) = N_k(F, \Dn)= \sset{g(0)}$ and for $\e_0>0$
there exists $\eta>0$ such that
\begin{equation}
\label{eq:omgdn}
\Omk(Z^{(\e)}, D) \subset [g(0)-\eta, g(0)+\eta]
\end{equation}
for all $0\leq\e<\e_0$. In particular, if $D \subset \Dn$ then~$\eta$ 
does not depend on~$F$.

Second, the following lemma implies that the dynamics on the 
canonical invariant region are independent of $g$ in a neighborhood of
zero. To highlight the dependence of~$F$ on~$g$, we write~$F^{(g)}$ for 
the remainder of this section.

\begin{lem}\label{lem:Coupling}
Let $A\subset \C$, with~$\C$ as in \eqref{eq:CIR}, be compact for the 
coupled phase oscillator system~\eqref{eq:OscVF} with coupling 
function~$g$. Then there exists a closed interval $I\subset (0, 2\pi)$ 
such $F^{(g)}|_A = F^{(\hat{g})}|_A$ for all coupling functions~$\gh$ 
with $\gh|_{I} = g|_{I}$.
\end{lem}

\begin{proof}
Since $A\subset\C$ note that 
$\abs{\vphi_j-\vphi_k}>0$ for $k\neq j$ and
$(\vphi_1, \dotsc, \vphi_\maxdim)\in A$.
Set 
$J = \bigcup_{j\neq k}\set{\vphi_k-\vphi_j}{\vphi\in A} \subset \Tor\sm\sset{0}$
and we may write $J\subset (0, 2\pi)$.
Since~$A$ is compact, $I = [\inf J, \sup J]$ is the desired compact 
interval. Now $F^{(g)}|_A$ only depends on the values $g$ takes
on~$I$ and the result follows.
\end{proof}

\section{Persistence of weak chimeras}
\label{sec:CWC}

We now apply the results of Section~\ref{sec:Persistence} to networks with weakly coupled populations of phase oscillators, where each population is as introduced in the previous section.

\subsection{Persistence for weakly symmetrically coupled populations}
The weakly coupled product of the dynamics~\eqref{eq:OscVF} with itself which defines a dynamical system with 
$\vphi = (\vphi_{1}, \vphi_{2})\in\Torn\times\Torn=\Tornn$ where $\vphi_{\ell}=(\vphi_{\ell,1}, \dotsc, \vphi_{\ell,\maxdim})$. 
More explicitly, the dynamics in the case of weak coupling are given by
\[
\dot{\vphi}=F^{(\e,g)}(\vphi)
\]
where
\begin{equation}\label{eq:DynP}
\begin{aligned}
\dot\vphi_{1,k} &= F^{(\e,g)}_{1,k}(\vphi) := F_k(\vphi_{1}) + \e \Th_k(\vphi_{1},\vphi_{2}) ,\\
\dot\vphi_{2,k} &= F^{(\e,g)}_{2,k}(\vphi) := F_k(\vphi_{2}) + \e \Th_k(\vphi_{2},\vphi_{1}) ,
\end{aligned}
\end{equation}
for $k=1, \dotsc, \maxdim$ where $F_k$ is given by \eqref{eq:OscVF} 
and~$\Theta$ by \eqref{eq:Theta}. If~$A_\e\subset\Tornn$ is compact 
and dynamically invariant for~\eqref{eq:DynP} we let 
$\Omega_{\ell,k}(F^{(\e,g)}, A_\e)$ denote the angular frequency intervals 
for oscillator~$(\ell, k)$ with phase $\vphi_{\ell,k}$, 
$\ell\in\sset{1, 2}$, $k\in\sset{1, \dotsc, \maxdim}$.
Moreover, for $D\subset\Tornn$ write 
\[N_{\ell, k}(F^{(\e, g)}, D) = \Big[\inf_{\vphi\in D}F^{(\e, g)}_{\ell, k}(\vphi), \sup_{\vphi\in D}F^{(\e, g)}_{\ell, k}(\vphi)\Big]\] as in~\eqref{eq:Bound}
and we have
$\Omega_{\ell,k}(F^{(\e,g)}, A_\e)\subset N_{\ell, k}(F^{(\e, g)}, A_\e)$ 
if~$A_\e$ is dynamically invariant.
Observe that for $\e=0$ the system decouples into two identical 
groups of~$n$ oscillators---both of which with nontrivial 
dynamics~\eqref{eq:OscVF}.

Note that, in addition to the continuous~$\Tor$ symmetry, \eqref{eq:DynP} is equivariant with respect to the action of a symmetry group $\Sn\wr \Sk{2}$ where~$\wr$ is the wreath product~\cite{Dionne1996}. That is, $\Sn\wr\Sk{2}= (\Sn)^2\times_s \Sk{2}$ where the~$\Sn$ permute the oscillators within the each group of $n$ oscillators and the $\Sk{2}$ permutes the two groups. Observe that this is only a semidirect product~$\times_s$ as the two sets of permutations do not necessarily commute. This group acts transitively on the oscillators: 
the oscillators are indistinguishable. Both
$\Torn\times\Dn\subset\Tornn$ and $\Dn\times\Torn\subset\Tornn$ are dynamically invariant for any $\e\geq 0$ as fixed point subspaces of the action of $\Sn\wr\Sk{2}$.

\subsubsection{Persistence of weak chimeras}
We now state the main result of this section; the notation~$\Ai$ 
suggests that this invariant set corresponds to the cluster of incoherent oscillators for the chimera while the remaining oscillators are coherent in the sense that they are fully phase-synchronized.

\begin{thm}
\label{thm:CWC}
Suppose that~$g$ is a coupling function such that $\Ai\subset\C$ is a compact, forward invariant, and sufficiently stable (or unstable) set for the dynamics of~\eqref{eq:OscVF}. For any $\delta>0$ with $\overline{B_\delta(\Ai)}\subset\C$ there exist a smooth coupling function $\hat{g}$ and an $\e_0>0$ such that the
weakly coupled product system~\eqref{eq:DynP} for $F^{(\e,\hat{g})}$ has a weak chimera~$A_\e$ with $A_\e \subset \Dn\times B_\delta(\Ai)$ for all $0\leq\e<\e_0$.
\end{thm}

\begin{proof}
Let $\delta>0$ be given and let us assume~$\Ai$ is sufficiently stable.
Compactness and continuity imply that 
$M := \max_{k\in\sset{1, \dotsc, \maxdim}}\max_{\vphi\in\Tornn}\abs{\Th_k(\vphi)}<\infty$.
Define
\[
N(g) := \bigcup_{k=1}^{\maxdim} N_{2,k}(F^{(0, g)}, \Dn \times B_\delta(\Ai)),
\]
assume that~$\e_0$ is fixed (we will determine the exact value below)
and set $\eta=\e_0M$.
If $B_{\eta}(\sset{g(0)})\cap B_{\eta}(N(g)) = \emptyset$ 
then we choose $\gh=g$, otherwise we choose~$\gh$ by Lemma~\ref{lem:Coupling}
such that 
$F^{(0, \gh)}|_{B_\delta(\Ai)^2}=F^{(0, g)}|_{B_\delta(\Ai)^2}$---implying $N(g)=N(\gh)$---and $\gh(0)$ sufficiently large so that
$B_{\eta}(\sset{\gh(0)})\cap B_{\eta}(N(\gh)) = \emptyset$.
In particular, for this choice of~$\gh$ the set~$\Ai$ is still 
forward invariant, and sufficiently stable in $B_\delta(\Ai)$
for~\eqref{eq:OscVF}. A similar argument applies if~$\Ai$ is 
sufficiently unstable.

Applying Theorem~\ref{thm:PersistenceFactor} (note that $\Delta_n$ is compact in the topology of $\Tor^n$) yields 
an~$\e_0>0$---take this as the unspecified value of~$\e_0$ above---and 
compact sets~$A_\e\subset\Dn\times B_\delta(\Ai)$ for all $0\leq\e<\e_0$ 
such that~$A_\e$ is dynamically invariant for the flow defined 
by~$F^{(\e, g)}$. Moreover, $A_\e$ can be assumed to be connected and chain-recurrent buy taking a subset if necessary.

It remains to be shown that any~$A_\e$ is a weak chimera. For
$\e=0$ the dynamics are uncoupled with 
$\Omega_{1,k}(F^{(0,\gh)}) = N_{1,k}(F^{(0, \gh)}, \Dn \times \Ai) = \sset{\gh(0)}$
for $k\in\sset{1, \dotsc, \maxdim}$. The weak coupling for $\e>0$
in~\eqref{eq:DynP} may be seen as a bounded nonautonomous perturbation
to each factor, that is, $\dot\vphi_{\ell}=F(\vphi_{\ell})+\e Y(\vphi_{\ell}, t)$ 
with $\norm{Y}\leq M$. By Lemma~\ref{lem:CloseCoarseBound}
we have (with~$\eta$ as above)
\begin{align*}
\Om_{1,k}(F^{(\e,\gh)}, A_\e) &\subset 
B_\eta(N_{1,k}(F^{(0, \gh)}, \Dn \times B_\delta(\Ai))) = B_\eta(\sset{\gh(0)}),\\
\Om_{2,k}(F^{(\e,\gh)}, A_\e) &\subset B_\eta(N_{2,k}(F^{(0, \gh)}, \Dn \times B_\delta(\Ai))) \subset B_\eta(N(\gh))
\end{align*}
for all $k\in\sset{1, \dotsc, \maxdim}$ and any $0\leq \e<\e_0$.
For $\vphi\in A_\e$ we have $\vphi_{1, 1}=\dotsb=\vphi_{1, \maxdim}$
and hence 
$\Om_{1,1}(F^{(\e,g)}, A_\e)=\dotsb=\Om_{1,\maxdim}(F^{(\e,g)}, A_\e)$.
Since $B_\eta(\sset{\gh(0)})\cap B_\eta(N(\gh))=\emptyset$
by the choice of~$\gh$, we have $\Om_{1,k}(F^{(\e,g)}, A_\e)\cap
\Om_{2,k}(F^{(\e,g)}, A_\e)=\emptyset$ for all 
$k\in\sset{1, \dotsc, \maxdim}$. Therefore any~$A_\e$ is a weak chimera 
for~\eqref{eq:DynP} with the coupling function~$\gh$.
\end{proof}

\begin{rem}
We remark that the results of Theorem~\ref{lem:Coupling} can be generalized in a straightforward way to $m\geq 2$ populations of $n$ coupled phase 
oscillators with dynamics given by 
\begin{equation}\label{eq:DynPnm}
\dot\vphi_{\ell,k} = F_k(\vphi_{\ell}) + \e\sum_{r\neq \ell}\Th_k(\vphi_{\ell}, \vphi_r)
\end{equation}
for phases $\vphi_{\ell,k}$ with $k=1, \dotsc, \maxdim$ and $\ell=1, \dotsc, m$. This more general system has symmetry $\Sn\wr \Sk{m}$ which acts transitively on the oscillators.
\end{rem}

\subsubsection{Stability}

The stability of a weak chimeras~$A_\e$ of Theorem~\ref{thm:CWC} as a subset of $\Dn\times\Torn$ depends on the stability properties of~$\Ai\subset\Torn$. Introduction of coordinates $\psi_{\ell,k}=\vphi_{\ell,k}-\vphi_{1,k}$ eliminates the phase shift symmetry. In fact $\psi_{1,k} = 0$ and thus the reduced system is a dynamical system on~$\Torn$. If~$\Ai$ is sufficiently stable so is~$A_\e$ in the reduced system; if~$\Ai$ is sufficiently unstable in the unperturbed system so is~$A_\e$. Of course, if~$\Ai$ is sufficiently stable we obtain a sufficiently unstable set by reversing time and vice versa.

Transversal stability of the invariant set $\Dn\times\Torn\subset\Tornn$ is determined by the sign of~$g'(0)$ \cite{Ashwin1992}. More precisely, $\Dn\times\Torn\subset\Tornn$ is asymptotically stable if $g'(0)<0$ and asymptotically unstable if $g'(0)>0$. Through local perturbation of the coupling function, one can get the desired transversal stability properties.

\begin{cor}
\label{cor:StabilitySwap}
Let $A_\e\subset\Dn\times\Torn$ be a weak chimera of~\eqref{eq:DynP} with coupling function~$g$ such that~$A_\e$ is asymptotically stable in $\Dn\times\Torn$. Then there exists a coupling function such that~$A_\e$ is asymptotically stable in~$\Tornn$.
\end{cor}

\begin{proof}
An application of Lemma~\ref{lem:Coupling}; choose a coupling function~$\gh$ such that 
$F^{(0,\gh)}|_{\Dn\times\overline{B_\delta(\Ai)}}=F^{(0,g)}|_{\Dn\times\overline{B_\delta(\Ai)}}$ and 
$\gh'(0)<0$.
\end{proof}

Consequently, the weak chimera states constructed above can be asymptotically stable, asymptotically unstable, or of saddle type. They are asymptotically stable (unstable) if~$\Ai$ is sufficiently stable (unstable) and $\Dn\times\Torn$ is transversally stable (unstable) and of saddle type if
if~$\Ai$ is sufficiently stable and $\Dn\times\Torn$ is transversally unstable and vice versa.

\subsubsection{Chaotic weak chimeras: dynamics on the invariant set}

Suitable coupling functions may now give rise to chaotic weak chimera 
in the limit of vanishing coupling. So far, we have not made any 
assumptions on the dynamics on~$\Ai$; they may be chaotic. If~$g$ is 
chosen such that the dynamics of \eqref{eq:OscVF} on~$\Ai$ has 
positive Lyapunov exponents then so will the dynamics of~\eqref{eq:DynP} 
on $A_0=\Dn\times\Ai$. Therefore it is reasonable to assume that
the positive Lyapunov exponents will persist for a set of $\e>0$
of positive measure, that is~$A_\e$ are chaotic weak chimeras 
for~$\e$ sufficiently small. 

\begin{rem}
\label{rem:PersistenceChaos}
To rigorously show that positive Lyapunov exponents persist for nonzero $\e>0$ one would have to make further assumptions that guarantee there is a suitable invariant measure for $\e=0$ that deforms nicely for sufficiently small $\e>0$; see also Section~\ref{sec:Persistence}.
\end{rem}

We give explicit examples of coupling functions that give rise to
chaotic dynamics on~$A_0$ and numerical evidence that the persisting
weak chimeras~$A_\e$ are chaotic for nonzero coupling
in Section~\ref{sec:Numerics}.

\subsection{Persistence for coupling breaking symmetry}
\label{sec:CWCnosym}

So far we have considered weak coupling between populations that preserved the symmetry of the system for any choice of~$\e$. Using the notation otherwise as for \eqref{eq:DynP} and assume that $Y_1, Y_2: \Tornn\to \Rn$ are Lipshitz continuous. The system 
\begin{equation}\label{eq:DynPns}
\begin{aligned}
\dot\vphi_{1,k} &= F^{(\e,g)}_{1,k}(\vphi) = F_k(\vphi_1) + \e Y_{1,k}(\vphi),\\
\dot\vphi_{2,k} &= F^{(\e,g)}_{2,k}(\vphi) = F_k(\vphi_2) + \e Y_{2,k}(\vphi),
\end{aligned}
\end{equation}
for $k=1, \dotsc, \maxdim$ is a weakly coupled product system. While~\eqref{eq:DynPns} is $(\Sn\wr\Sk{2})\times\Tor$ equivariant for $\e=0$, 
for a generic choice of $Y_1, Y_2$ the symmetry is broken whenever $\e>0$\footnote{Note that in this case the oscillators may not be indistinguishable anymore.}. However, the 
weak chimera may persist as the next result shows.

\begin{thm}
\label{thm:CWCns}
Suppose that~$g$ is a coupling function such that $\Ai\subset\C$ and $\Ac = \Dn\subset\Torn$ are compact, forward invariant and sufficiently stable (or sufficiently unstable) sets for the dynamics of~\eqref{eq:DynPns} for $F^{(0,g)}$. For any $\delta>0$ there exist a smooth coupling function~$\gh$ and $\e_0>0$ such that the weakly coupled product system~\eqref{eq:DynP} has a sufficiently stable (unstable) weak chimera~$A$ with $A \subset B_\delta(\Ac\times\Ai)$ for $F^{(\e,\gh)}$ and any $0\leq\e<\e_0$.
\end{thm}

\begin{proof}
Persistence of invariant sets~$A_\e$ in each factor follows from 
Theorem~\ref{thm:Products}. The same argument as in the proof of 
Theorem~\ref{thm:CWC} shows that there is a smooth coupling function
such that~$A_\e$ is a weak chimera since Lemma~\ref{lem:Coupling} 
allows us to modify the coupling function in an open neighborhood of zero.
\end{proof}

While Theorem~\ref{thm:CWCns} only gives existence of sufficiently stable (or sufficiently unstable) weak chimeras for systems~\eqref{eq:DynPns} there may be others of saddle type. Note also that there are coupling functions for which there is a Lyapunov function for the fully synchronized solution~\cite{VanHemmen1993}.

\section{Examples of chaotic weak chimeras}
\label{sec:Numerics}

In this section we explicitly construct coupling functions that give rise 
to chaotic weak chimeras for~\eqref{eq:DynP} with $\e=0$. Moreover, we 
demonstrate numerically that these chaotic weak chimeras persist for $\e>0$. 
While we use the methods developed in the previous section to 
construct the chaotic weak chimeras, we do not check rigorously whether
the assumptions are satisfied.

\subsection{A chaotic weak chimera of saddle type}

Consider the dynamics of~\eqref{eq:OscVF} for $\maxdim=4$ oscillators. Suppose that the coupling function is given by the truncated Fourier series
\begin{equation}
\label{eq:gchaos}
{g}(\phi) = \sum_{r=0}^{4} c_r \cos (r\phi+\xi_r)
\end{equation}
with $c_1 = -2$, $c_2 = -2$, $c_3 = -1$, and $c_4 = -0.88$. For $\xi_1 = \eta_1$, $\xi_2 = -\eta_1$, $\xi_3 = \eta_1+\eta_2$, and $\xi_4 = \eta_1+\eta_2$ with $\eta_1=0.11151$, $\eta_2=0.05586$ the function~${g}$ gives rise to a chaotic attractor $\Ai\subset\C$ with positive maximal Lyapunov 
exponents~\cite{Bick2011}.

\begin{figure}
{\centerline{\includegraphics[scale=\imagescaling]{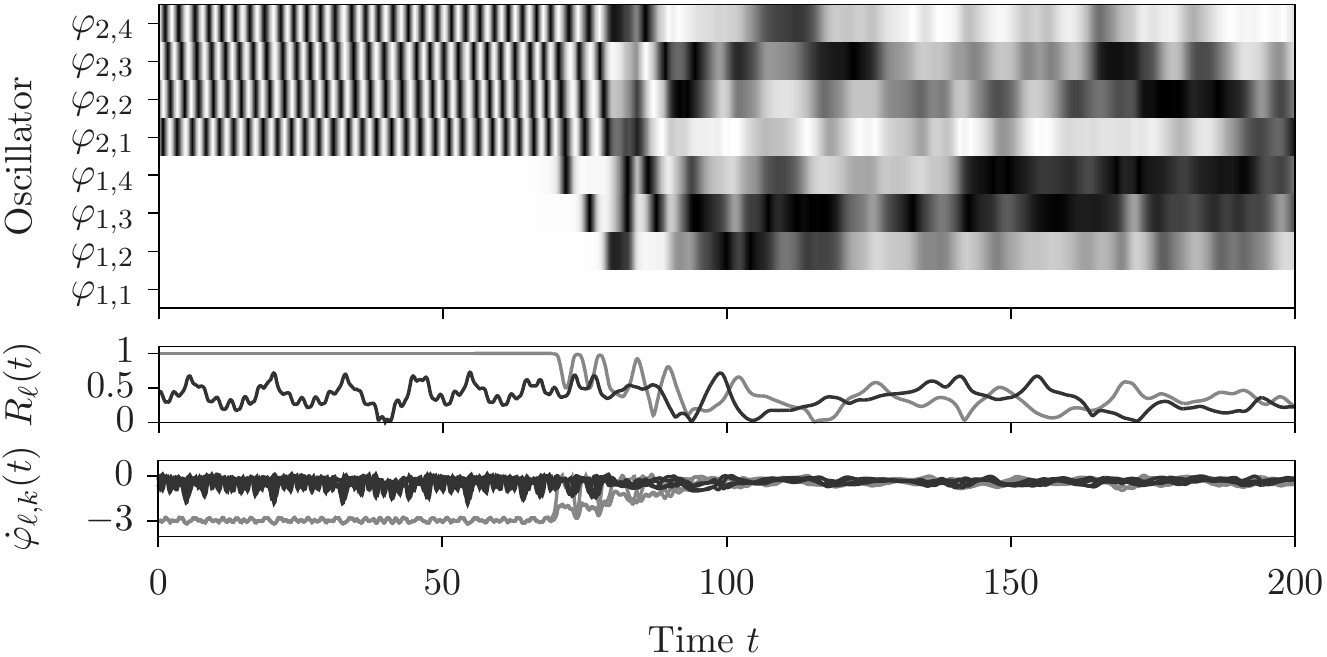}}}
\caption{\label{fig:N4CWCSaddle}
A transient started near a chaotic weak chimera of saddle type for two populations of $\maxdim=4$ oscillators \eqref{eq:DynP} with coupling function~\eqref{eq:gchaos} for $\eta_1=0.1104$, $\eta_2=0.057511$ and $\e=0.2$. The phase evolution of each oscillator in a co-rotating frame at the speed of oscillator~$1$ is shown using
a periodic grey scale ($\vphi_{\ell,k}(t)=0$, in black~$\vphi_{\ell,k}(t)=\pi$ in white). 
Middle and bottom panels show the order parameters~$R_\ell(t)$ of the 
populations~$\vphi_{\ell}$, $\ell=1,2$, and the instantaneous frequencies~$\dot\vphi_{\ell,k}(t)$. Note that 
for $t\lessapprox 70$ the populations are clearly are not frequency 
synchronized with chaotic oscillations---evidence of a chaotic weak chimera. 
After the trajectory leaves the vicinity of the saddle, it converges to
a homogeneous chaotic state.\bigskip
}
\end{figure}

\begin{figure}
\centerline{\includegraphics[scale=\imagescaling]{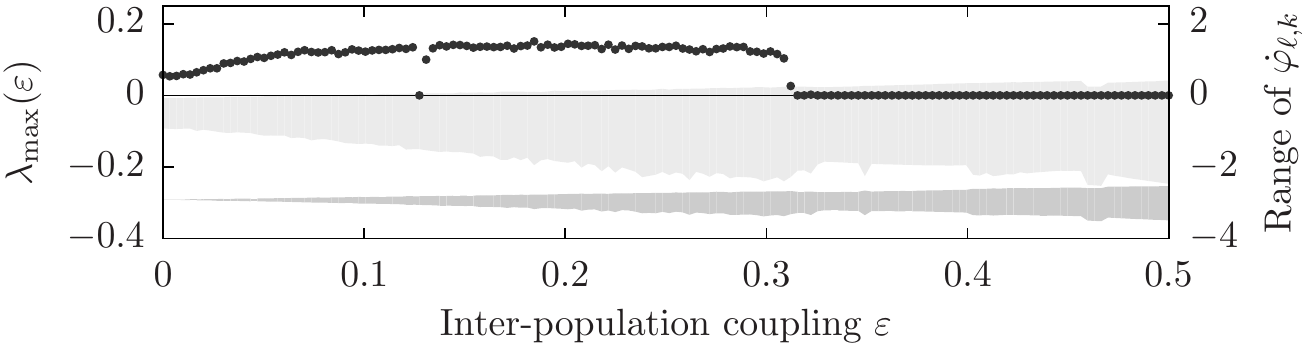}}
\caption{\label{fig:N4CWCSaddleScan}
Positive maximal Lyapunov exponents persist for $\e>0$ for two populations 
of $\maxdim=4$ oscillators~\eqref{eq:DynP} with coupling 
function~\eqref{eq:gchaos} for $\eta_1=0.1104$, $\eta_2=0.057511$. We 
approximated the maximal Lyapunov exponent (black dots) by 
integrating~\eqref{eq:DynP} from a random initial condition (sampled uniformly on 
$\Dn\times\Torn$) for $T=7000$ time units. 
The shaded regions show the intervals $[\min_{k,t}\dot\vphi_{\ell,k}(t), \max_{k,t}\dot\vphi_{\ell,k}(t)]$ for 
$\ell=1$ (dark grey) and 
$\ell=2$ (light grey)---where these do not overlap, there is no frequency synchronization between the two populations and hence a weak chimera.
}
\end{figure}

For this choice of coupling function and suitable~$\e$, the product system~\eqref{eq:DynP} gives rise to a chaotic weak chimera of saddle type. While 
these chaotic weak chimeras are attracting in $\Dn\times\Torn$, they are transversally unstable in~$\Tornn$ since ${g}'(0)>0$. 
Figure~\ref{fig:N4CWCSaddle} shows a trajectory that is initialized slightly 
off the invariant set.
The population order parameter $R_\ell(t) = \tabs{\frac{1}{\maxdim}\sum_{j=1}^\maxdim \exp(i\vphi_{\ell,j})}$, $\ell=1, 2$, characterizes synchronization in the system is equal to one for the synchronized population and fluctuating chaotically for the incoherent population.

Numerical simulations show that these chaotic weak chimeras persist for $\e>0$. 
We calculated the maximal Lyapunov exponent~$\lmax$ by numerically integrating~the variational equations for \eqref{eq:DynP}\footnote{Integration for $T=7000$ time units was carried out in MATLAB using the standard adaptive Runge--Kutta scheme with relative and absolute error tolerances of~$10^{-9}$ and~$10^{-11}$ respectively.}. As shown in Figure~\ref{fig:N4CWCSaddleScan}, chaotic weak chimeras apparently exist for most values $\e\lessapprox 0.3$.

\subsection{An attracting chaotic weak chimera}

As indicated in Corollary~\ref{cor:StabilitySwap} the chaotic weak chimeras of saddle type can be made attracting in~$\Tornn$ with a suitable local perturbation to the coupling function~$g$. Define
\[
\Bm(x) := \begin{cases}\exp\fleft(-\frac{1}{1-x^2}\right)&\text{if} -1<x<1,\\
0& \text{otherwise}\end{cases}
\]
and let $a\in\R$, $b \in (0, \pi)$, $c\in \Tor$ be parameters. Consider the ``bump function'' $\Bm_{abc}(\phi) = a \Bm\big(\frac{\phi}{b}-c\big)$ 
with~$\phi$ taken modulo~$2\pi$ with values in $(-\pi, \pi]$. Thus, $\Bm_{abc}(\phi)$ is a $2\pi$-periodic~$\Cinf$ function. Define
\begin{equation}
\label{eq:ghat}
\gh := {g} + \Bm_{abc}.
\end{equation}
Let~$D$ be an open set with $\Ai\subset D\subset\C$. Now choose $a, b, c$ such that $F^{(\e, g)}|_D=F^{(\e, \gh)}|_D$ and $\gh'(0)<0$. The dynamics of~\eqref{eq:DynP} with coupling function~$\gh$ give rise to attracting chaotic weak chimeras.
Figure~\ref{fig:CinftyN4CWC}(\textsc{a}) shows a single trajectory for $\e=0.2$ that is initialized slightly 
off the invariant set.

\begin{figure}
\subfloat[Populations of $\maxdim=4$ oscillators]{\centerline{\includegraphics[scale=\imagescaling]{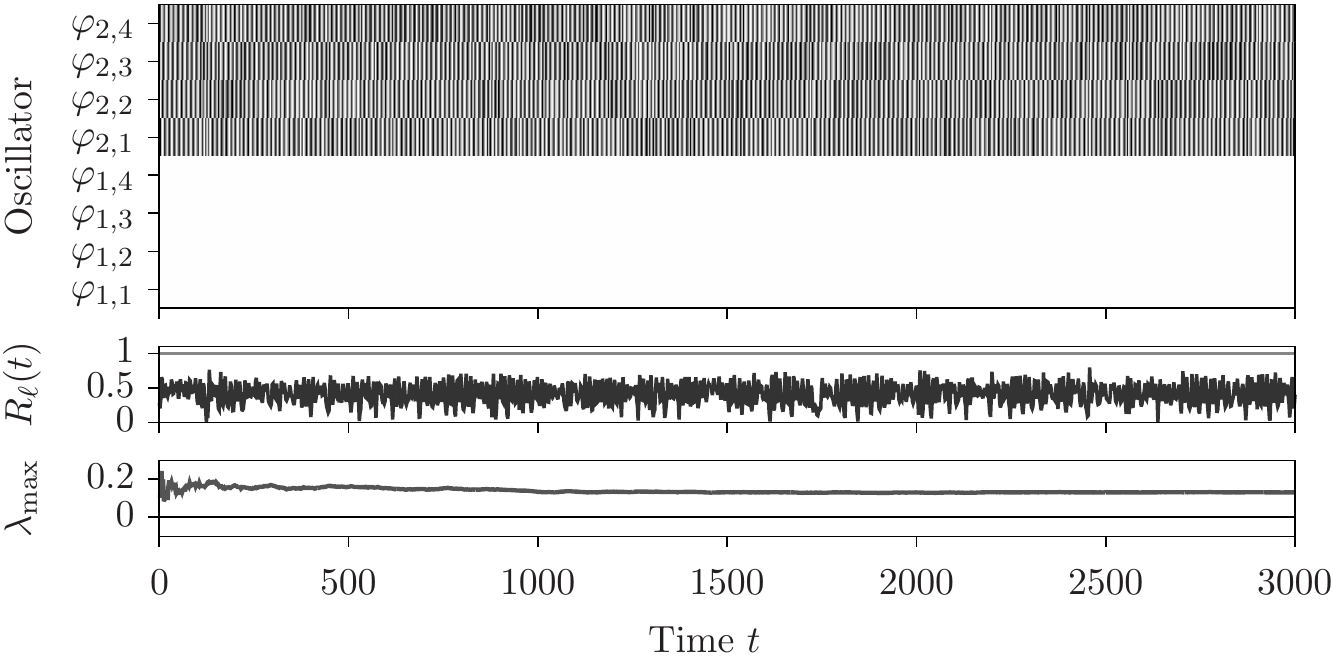}}}\\
\subfloat[Populations of $\maxdim=7$ oscillators\label{subfig:test}]{
\centerline{\includegraphics[scale=\imagescaling]{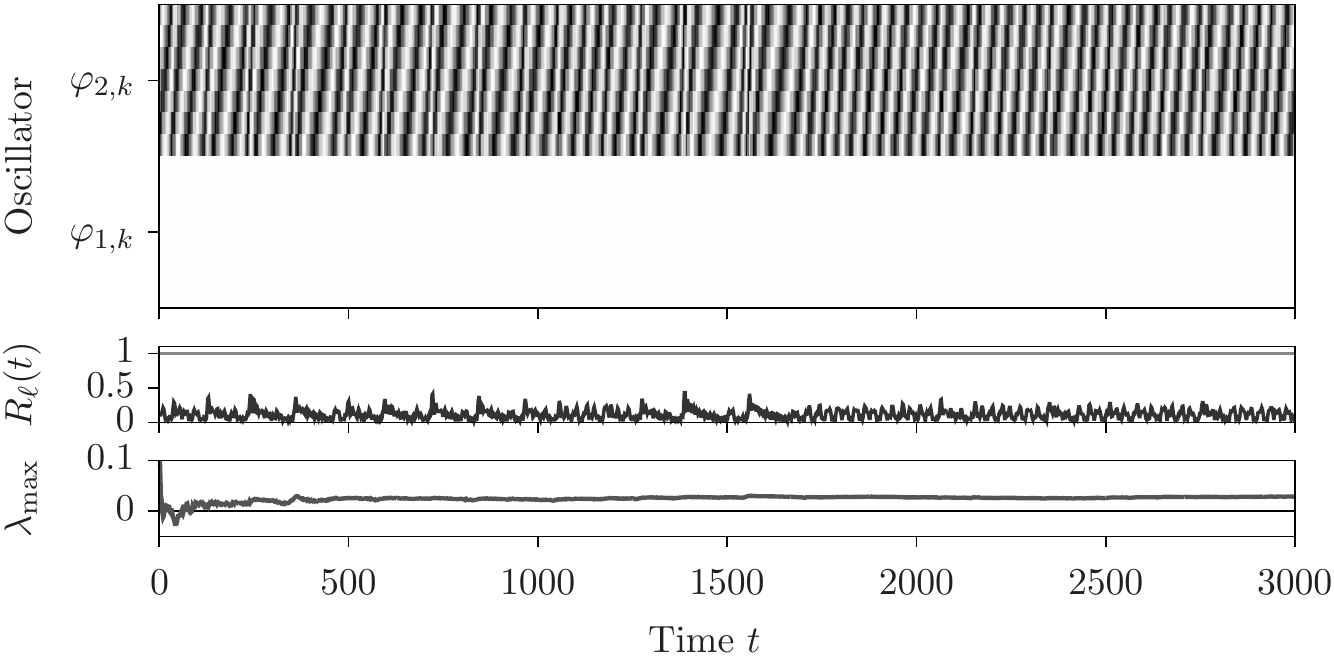}}
}
\caption{\label{fig:CinftyN4CWC}
Attracting chaotic weak chimera in two populations of oscillators as for Figure~\ref{fig:N4CWCSaddle} but with coupling function~$\gh$
with parameters $a=1$, $b=0.1$, $c=0.04$, $\eta_1=0.1104$, and $\eta_2=0.057511$. Note that $\gh(0)=-6.1332$ and $\gh'(0)=-2.5782$ indicating bistability of the fully synchronized state and the chaotic incoherent state for each uncouple cluster. The top panels shows the phases for $\e=0.2$ in greyscale for a co-rotating frame, the middle panels shows the order parameters~$R_\ell(t)$ of the two populations and convergence to the maximal Lyapunov exponent~$\lmax$ is shown in the bottom panels.
}
\end{figure}

Attracting chaotic weak chimeras also appear for larger population sizes. 
With the same coupling function $\gh$ as for $\maxdim=4$ oscillators we 
also find chaotic weak chimeras for $\maxdim\in\sset{5, 7}$; cf.~\cite{Bick2011}. A trajectory for $\maxdim=7$ oscillator in each population is depicted in Figure~\ref{fig:CinftyN4CWC}(\textsc{b}).

\subsection{Other coupling functions giving chaotic weak chimeras}
\label{subsec:trigpoly}

\begin{figure}
\subfloat[$\maxdim=4$ oscillators]{
\centerline{\includegraphics[scale=\imagescaling]{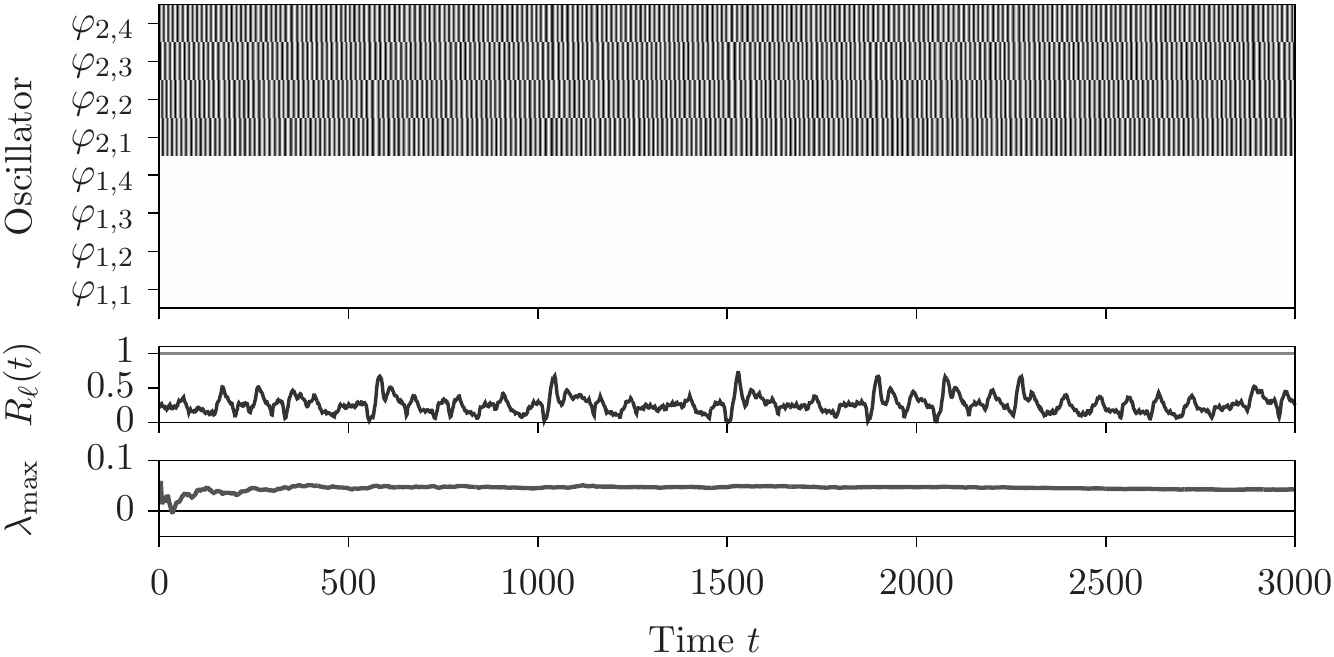}}
}\\
\subfloat[Maximal Lyapunov exponents for varying~$\e$]{
\centerline{\includegraphics[scale=\imagescaling]{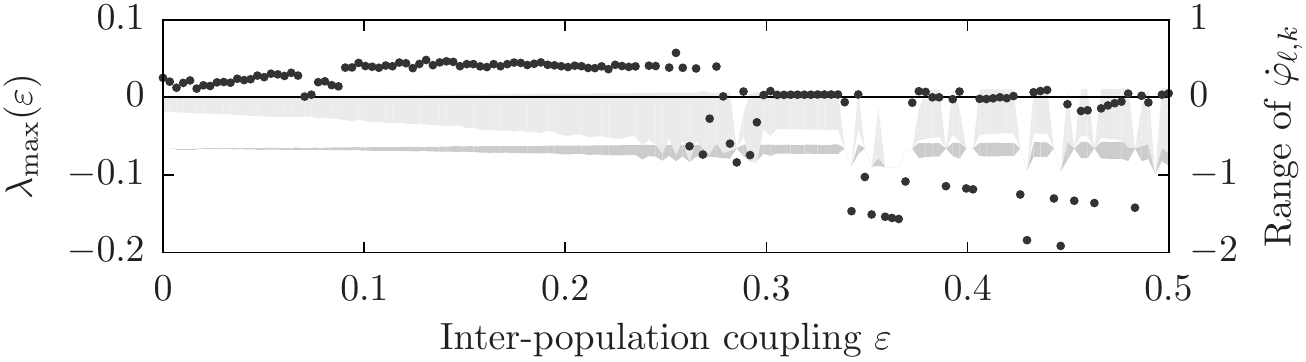}\quad}
}\\
\caption{\label{fig:ComegaN4CWC}
Chaotic weak chimeras appear for two populations of $\maxdim=4$ oscillators with the analytic coupling function with Fourier coefficients given in Table~\ref{tab:FourierCoeafficients}. Panel~(\textsc{a}) shows the dynamics for $\e=0.1$; the phase evolution of each oscillator is shown in a co-rotating frame at the 
speed of a synchronized oscillator in the first panel, the order parameters~$R_\ell(t)$ of each population in the second panel and the maximal Lyapunov exponent~$\lmax$ in the bottom panel. Panel~(\textsc{b}) shows positive maximal Lyapunov exponent~$\lmax(\e)$ for fixed initial conditions chosen on the attractor in the top panel using $T=7000$ time units to approximate the exponent and the range of $\dot\vphi_{\ell,k}$; cf.~Figure~\ref{fig:N4CWCSaddleScan}.
}
\end{figure}

The appearance of attracting chaotic weak chimeras is not limited to the perturbed coupling function constructed above; these typically have infinitely many nontrivial coefficients in their Fourier expansion. Attracting chaotic weak chimeras can also be found for trigonometric polynomial coupling functions with only finitely many nontrivial Fourier modes. Consider the dynamics of~\eqref{eq:DynP} with two populations of $\maxdim=4$ phase oscillators and coupling function 
\[
g(\phi) = \sum_{r=1}^{L} c_r \cos (r\phi)+ s_r\sin(r\phi)
\] 
with $L=11$ and Fourier coefficients as specified in Table~\ref{tab:FourierCoeafficients} in Appendix~\ref{app:TrigPoly}. This coupling function is analytic and the dynamics show that there is an attracting chaotic weak chimera for a range of $\e$; cf.~Figure~\ref{fig:ComegaN4CWC}(\textsc{a}). Our simulations indicate that apart from the chaotic weak chimera there are other stable attractors in the system. Further numerical investigation, using the fixed initial conditions on the chaotic weak chimera in Figure~\ref{fig:ComegaN4CWC}(\textsc{a}), is summarized in Figure~\ref{fig:ComegaN4CWC}(\textsc{b}). Positive maximal Lyapunov exponents can arise for a range of positive coupling values $\e\lessapprox 0.15$. It is possible that adiabatic continuation---that is,
using a point on the attractor for one parameter value as an initial condition for
a nearby parameter value---may give 
more detailed insights into how the chaotic weak chimeras develop and 
bifurcate as~$\e$ is varied.


\section{Discussion}
\label{sec:discuss}

In Section~\ref{sec:CWC} we prove a general existence result for weak chimeras in coupled phase oscillator systems. Chimeras are constructed in ``modular'' networks that are weakly coupled but the numerical investigations indicate that they exist even beyond the weak coupling limit. As dynamically invariant sets, the weak chimeras are not transient---``persistent'' chimera states in systems with generalized coupling were recently observed numerically~\cite{Suda2015}. But the existence of weak chimeras of saddle type induces transient dynamics for initial conditions close to the
saddle.

While our results are stated for the symmetric case that each population has the same number of oscillators, it is straightforward to generalize the construction to asymmetric systems with populations of differing sizes, say of~$n_1$ and~$n_2$ oscillators. If the oscillators belonging to the 
coherent region are not only frequency synchronized but also phase locked (their phase difference stays constant over time) the minimal number of oscillators to get similar chaotic weak chimera dynamics is $n_1+n_2=5$ since the dynamics are effectively three dimensional; however in this case the oscillators will not longer be indistinguishable.

The weak chimeras of saddle type (i.e., with both stable and unstable directions) are certainly of interest as analogues to the states studied 
in~\cite{Wolfrum2011b}. Further existence results of such chaotic weak chimeras may be possible with the methods mentioned in Section~\ref{sec:Persistence}: normally hyperbolic sets persist under small perturbations. Moreover, for suitable choice of parameters, such chimera saddles could have connecting orbits, thus leading to transitions from one chimera state to the next one. Constructions of such connections could be seen as a form of control of spatially localized dynamics similar to chimera control~\cite{Bick2014a}.

The relationship between the chaotic weak chimeras we constructed here and ``classical'' chimera states in multiple populations needs to be clarified further: in~\cite{Abrams2004} the coupling function only contains a single nontrivial Fourier mode and the coupling strength between the populations is almost as strong as within a population. Moreover, ``large'' populations are considered in the continuum limit of infinitely many oscillators where the dynamics reduce to mean field equations~\cite{Abrams2004, Panaggio2015a}. By contrast, the chaotic weak chimeras here arise in weakly coupled populations of four to seven oscillators with coupling given by functions with four or more nontrivial Fourier components. Continuation of these solutions in a suitable parameter space may shed some light on whether there these solutions are directly related. It is worth noting that the chaotic weak chimeras constructed here in such small networks also exhibit chaotic fluctuations of the order parameter~\cite{Bick2011} similar what is observed for large ensembles of interacting nonsmooth oscillators~\cite{Pazo2014}. Clearly, our results for small populations extend to more general (limit cycle) oscillators whose phase reduction is given by the coupling functions of our construction; see~\cite{Kori2008} for approaches to design general oscillator systems with a desired phase reduction. In fact, we expect chaotic dynamics to be more common for general oscillator networks that are not necessarily weakly coupled (as long as phases can still be defined, see remarks in~\cite{Bick2015d}) as amplitude variations facilitate chaotic dynamics in globally coupled identical oscillators~\cite{Nakagawa1993}.

In summary, the notion of a weak chimera yields a mathematically precise definition of chimera states for finite dimensional systems. Here we show that weak chimeras which mimic the dynamical behavior of regular chimeras can be constructed explicitly. We anticipate that a similar study 
of weak chimeras may give further insight into how chimeras arise in finite-dimensional coupled phase oscillator networks.

\section*{Acknowledgements}
CB and PA would like to thank Mike Field and Tiago Pereira for helpful discussions. The research leading to these results has received funding from the People Programme (Marie Curie Actions) of the European Union's Seventh Framework Programme (FP7/2007--2013) under REA grant agreement n\textsuperscript{o} 626111 (CB).

\bibliographystyle{plain}
{\scriptsize
\bibliography{citations_18_9} 
}

\appendix

\section{A trigonometric polynomial coupling function}
\label{app:TrigPoly}
Table~\ref{tab:FourierCoeafficients} gives coefficients for a trigonometric polynomial coupling function that gives an attracting chaotic weak chimera for the system discussed in Section~\ref{subsec:trigpoly}.

\begin{table}[h]
\begin{align*}
c_0 &= -0.48239 & s_0 &= 0.0538\\
c_1 &= -0.48239 & s_1 &= -0.05766\\
c_2 &= -0.23244 & s_2 &= 0.03754\\
c_3 &= -0.20325 & s_3 &= 0.0313\\
c_4 &= 0.01322 & s_4 &= -0.00626\\
c_5 &= 0.01261 & s_5 &= -0.0074\\
c_6 &= 0.01191 & s_6 &= -0.00849\\
c_7 &= 0.01111 & s_7 &= -0.00951\\
c_8 &= 0.01023 & s_8 &= -0.01045\\
c_9 &= 0.00927 & s_9 &= -0.01131\\
c_{10} &= 0.00823 & s_{10} &= -0.01209
\end{align*}
\caption{
\label{tab:FourierCoeafficients}
Fourier coefficients of a trigonometric polynomial coupling function that gives a chaotic
weak chimera for the eight oscillator system discussed in Section~\ref{subsec:trigpoly}.}
\end{table}

\end{document}